\documentclass[11pt,twoside]{article}
\usepackage{amsfonts}
\usepackage{amsmath,amssymb}
\usepackage{multicol}
\usepackage{graphics}
\usepackage{cite}
\usepackage{color} 

\newtheorem{theorem}{Theorem}[section]
\newtheorem{lemma}[theorem]{Lemma}

\newtheorem{example}[theorem]{Example}
\newtheorem{definition}[theorem]{Definition}

\numberwithin{equation}{section}
\newenvironment{proof}[1][Proof]{\noindent\textbf{#1.} }{\hfill $\Box$}

\allowdisplaybreaks \numberwithin{equation}{section}
\makeatletter
\begin{document}
\author{Li Zhang, Wan-Tong Li\thanks{Corresponding author (wtli@lzu.edu.cn).}~ and Zhi-Cheng Wang\\
School of Mathematics and Statistics, Lanzhou University, \\
Lanzhou, Gansu 730000, People's Republic of China}

\title{\textbf{Entire Solution in an Ignition Nonlocal Dispersal Equation: Asymmetric Kernel}}
\date{}
\maketitle

\begin{abstract}
This paper mainly focus on the front-like entire solution of a classical nonlocal dispersal equation with ignition nonlinearity.  Especially, the dispersal kernel function $J$ may not be symmetric here. The asymmetry of $J$ has a great influence on the profile of the traveling waves and the sign of the wave speeds, which further makes the properties of the entire solution more diverse. We first investigate the asymptotic behavior of the traveling wave solutions since it plays an essential role in obtaining the front-like entire solution. Due to the impact of $f'(0)=0$, we can no longer use the common method which mainly depending on Ikehara theorem and bilateral Laplace transform to study the asymptotic rates of the nondecreasing traveling wave and the nonincreasing one tending to 0, respectively, thus we adopt another method to investigate them. Afterwards, we establish a new entire solution and obtain its qualitative properties by constructing proper supersolution and subsolution and by classifying the sign and size of the wave speeds.


\end{abstract}


\section{Introduction}
\noindent

In this paper, we are concerned with the following classical nonlocal dispersal equation:
\begin{equation}\label{101}
u_{t}(x,t)=(J*u)(x,t)-u(x,t)+f(u(x,t)), ~~(x,t)\in\mathbb{R}^{2},
\end{equation}
where $(J*u)(x,t)-u(x,t)=\int_{\mathbb{R}}J(y)[u(x-y,t)-u(x,t)]dy$ is a nonlocal dispersal term, and $f$ is an ignition nonlinearity. They satisfy
\begin{description}
\item[(J1)] $J\in C(\mathbb{R})$, $J(x)\geq0$, $\int_{\mathbb{R}}J(x)dx=1$ and $\exists~\lambda>0$ such that $\int_{\mathbb{R}}J(x)e^{\lambda |x|}dx<\infty$.
 \item[(FI)] $f\in C^{2}(\mathbb{R})$, $f(0)=f(1)=0$, $f'(1)<0$, and there exists $\rho\in(0,1)$ such that
$f|_{[0,\rho]}\equiv0, f|_{(\rho,1)}>0$.
\end{description}

It is well known that the long range dispersal phenomenon is very widespread in natural life, thus in recent years,  nonlocal equations and systems have been widely concerned and studied by more and more researchers. Propagation dynamics is one of the hot research topics, more precisely, it includes  traveling wave solution, spreading speed, entire solution and so on.
Especially, entire solution has been widely studied because of its important theoretical and practical significance.  From the perspective of dynamical system, the solution of initial value problem can be seen as a semi-flow or a half-orbit, so we can only determine the state of the solution in $t\in[0,+\infty)$. But an entire solution can be viewed as a full-flow or a full-orbit, it allows us to study the information of the solution in any space point $x\in\mathbb{R}$ and any time $t\in\mathbb{R}$, which further helps us to understand the transient dynamics and the structure of the global attractors. For more rich entire solutions of classical Laplace diffusion and nonlocal dispersal equations and systems, one can refer to \cite{Hamel,GM,Crooks,LSW,li2008,LZZ,liu1,MN,Pan2009,wang2010,WLR2016,WR,ZLW,SLW} and references cited therein.

Note that above results for nonlocal equations and systems are all based on the symmetry of the kernel function. However, as we mentioned in the previous works \cite{SZLW, ZLWS}, the asymmetric dispersal phenomenon is very common and widespread in reality since there is a formal analogy between asymmetric-nonlocal-dispersal equations and reaction-advection-diffusion equations \cite{Coville1,SZLW}. For the entire solutions of reaction-advection-diffusion equations, we refer to \cite{Crooks,liu1} and references therein.

In \cite{SZLW, ZLWS}, we have investigated the entire solutions of the asymmetric dispersal equation \eqref{101} with monostable and bistable nonlinearities, respectively. Compared with the symmetric case \cite{LSW, SLW}, we found that it occurs new types of entire solutions due to the influence of the asymmetry of the kernel function. In addition, catching up phenomenon between traveling wave solutions is very common and qualitative properties of some entire solutions also become very different.

In this paper, we focus our attention on the entire solution of \eqref{101} with asymmetric kernel function and ignition nonlinearity since there is no result on this issue. We use front-like entire solution to describe the propagation of flame since different types of entire solutions corresponding to different ways of flame propagation. What should be pointed out is that the key step to construct such entire solution which behaves like interactions of different traveling waves is having a precise information on the asymptotic behaviors of the traveling wave solutions at infinity. Therefore, we take half the length of this paper to investigate them.

A traveling wave solution of \eqref{101} is a special translation invariant solution of the form $u(x,t)=\phi(\xi) (\xi=x+ct)$ which satisfies
\begin{equation}\label{102}
\begin{cases}
c\phi'(\xi)=J*\phi(\xi)-\phi(\xi)+f(\phi(\xi)),\\
\phi(-\infty)=0,\quad \phi(+\infty)=1,
\end{cases}
\end{equation}
where $\phi(\pm\infty)$ denotes the limit of $\phi(\xi)$ as $\xi\rightarrow\pm\infty$. For asymmetric kernel function $J$ and bistable or ignition nonlinearity, Coville \cite{Coville3} has proved the existence of the traveling wave solution $\phi$ and the uniqueness of the speed $c$. The specific results can be stated as follows. About the traveling wave solutions of the asymmetric equation \eqref{101} with monostable nonlinearity, one can refer to  Coville et al. \cite{Coville1}, Sun et al. \cite{SLW2} and Yagisita \cite{Ya2}.

\begin{theorem}{\rm(\cite{Coville3})}\label{th1}
Assume that $J$ satisfies the assumption (J1) and let $f$ be of bistable or ignition type. Then there exists a constant $c\in\mathbb{R}$ and a nondecreasing function $\phi$ such that $(\phi,c)$ is a solution of \eqref{102}. Moreover the speed $c$ is unique.
\end{theorem}

In virtue of the asymmetry of the kernel function $J$, we can not get a nonincreasing traveling wave solution directly, although we can do this for symmetric case. However, note that $\hat{u}(x,t):=u(-x,t)$ is a solution of the following nonlocal equation
\begin{equation}\label{103}
\hat{u}_{t}(x,t)=\int_{\mathbb{R}}J(y)\hat{u}(x+y,t)dy-\hat{u}(x,t)+f(\hat{u}(x,t))
\end{equation}
whenever $u(x,t)$ is a solution of \eqref{101}. Thus the existence of the nondecreasing traveling wave solution of \eqref{103} from 0 to 1 immediately implies the existence of the nonincreasing traveling wave solution of \eqref{101} from 1 to 0. Therefore, by using a similar method with that in \cite{Coville3}, we obtain a nonincreasing front $\hat{\phi}(\hat{\xi}):=\hat{\phi}(x+\hat{c}t)$ with speed $\hat{c}\in\mathbb{R}$ satisfying
\begin{equation}\label{104}
\begin{cases}
\hat{c}\hat{\phi}'=J*\hat{\phi}-\hat{\phi}+f(\hat{\phi}),\\
\hat{\phi}(-\infty)=1,\quad \hat{\phi}(+\infty)=0.
\end{cases}
\end{equation}
But here, the two waves $\phi(\xi)$ and $\hat{\phi}(\hat{\xi})$ are likely to be no longer symmetric in shape although they always are when the kernel function is symmetric.

In order to construct a proper supersolution and further obtain the existence and qualitative properties of a front-like entire solution of \eqref{101}, we must make clear the asymptotic behaviors of the traveling wave solutions $\phi(\xi)$ and $\hat{\phi}(\hat{\xi})$ at infinity.
As $\xi\rightarrow+\infty$ (or $\hat{\xi}\rightarrow-\infty$), since $f'(1)<0$, we can obtain a precise exponential decay rate of $\phi(\xi)$ (or $\hat{\phi}(\hat{\xi})$) by using similar arguments with those in \cite{ZLWS} for bistable case. However, as $\xi\rightarrow-\infty$ (or $\hat{\xi}\rightarrow+\infty$), due to the effect of $f'(0)=0$, we can not use the same method anymore, so by combining the methods provided in Coville \cite{Coville4} and Zhang \cite{ZGB}, we find another way which mainly based on the constant use of a comparison principle and the construction of appropriate barrier functions. But due to the asymmetry of the kernel function $J$, the details of the proof are very different. Moreover, in order to ensure $J$ satisfies a proper comparison principle, we give the following assumption, for this comment, one can refer to \cite{Coville2, Coville3} for more details.
\begin{description}
\item[(J2)] $\exists$~~$a\leq0\leq b, a\neq b$~~such that~~$J(a)>0$ and $J(b)>0$.
\end{description}
which equivalent to $\text{supp}(J)\cap \mathbb{R}^{+}\neq\varnothing$ and $\text{supp}(J)\cap \mathbb{R}^{-}\neq\varnothing$.

From now on, if there is no special note, we always assume the wave speeds $c\neq0$ and $\hat{c}\neq0$, in Section 2, we will give a special asymmetric kernel function as an example to illustrate that the assumption is reasonable and meaningful (Example \ref{em2.1}).

Now we state our main result.

\begin{theorem}\label{th3}
Assume that J satisfies (J1)-(J2) and \eqref{eq2.4}, f satisfies $\max_{u\in[0,1]}f'(u)<1$ and (FI). Let $\phi(x+ct)$ and $\hat{\phi}(x+\hat{c}t)$ be the monotone solutions of \eqref{102} and \eqref{104} which satisfy \eqref{eq4.1}, respectively. Then for any constant $\theta\in\mathbb{R}$,
\eqref{101} admits an entire solution $u(x,t): \mathbb{R}^{2}\rightarrow[0,1]$ satisfying
\begin{equation}\label{105}
\lim_{t\rightarrow-\infty}\left\{\sup_{x\leq-\frac{c+\hat{c}}{2}t}|u(x,t)-\hat{\phi}(x+\hat{c}t-\theta)|
+\sup_{x\geq-\frac{c+\hat{c}}{2}t}|u(x,t)-\phi(x+ct+\theta)|\right\}=0.
\end{equation}
And there exist positive constants $C_{1}$ and $C_{2}$ such that for any $\eta>0$,
\begin{equation}\label{106}
|u(x+\eta,t)-u(x,t)|\leq C_{1}\eta,~~~
\left|\frac{\partial u}{\partial t}(x+\eta,t)-\frac{\partial u}
{\partial t}(x,t)\right|\leq C_{2}\eta.
\end{equation}
In addition, $u(x,t)$ possesses different properties according to the sign and size of the speeds $c$ and $\hat{c}$ (see Lemma \ref{lem2.5}):
\begin{description}
\item[(a)] If $c>0, \hat{c}>0$, then $u(x,t)$ is increasing with respect to $\theta$ and for some constants $a,N, t_{0}\in\mathbb{R}$,
\begin{align}
&\lim_{\theta\rightarrow+\infty}u(x,t)=1  \text{~~uniformly for~~} (x,t)\in[-a,+\infty)^{2}\cup(-\infty,a]^{2},\label{eq109}\\
&\lim_{x\rightarrow+\infty}\sup_{t\geq t_{0}}|u(x,t)-1|=0,~~~
\lim_{x\rightarrow-\infty}\sup_{t\leq t_{0}}|u(x,t)-1|=0,\label{107}\\
&\lim_{t\rightarrow+\infty}\sup_{x\in[-N,+\infty)}|u(x,t)-1|=0.\label{108}
\end{align}
Moreover, the traveling wave solutions $\phi$ and $\hat{\phi}$ propagating in the same direction and from left of the x-axis as $t\rightarrow-\infty$. Since $c>\hat{c}$ (Lemma 2.5), the two waves getting closer and closer as time goes on, and finally $\phi$ is likely to catch up $\hat{\phi}$.
\item[(b)] If $c<0, \hat{c}<0$, then $u(x,t)$ is increasing with respect to $\theta$ and
\begin{align*}
&\lim_{\theta\rightarrow+\infty}u(x,t)=1  \text{~uniformly for~} (x,t)\in[-a,+\infty)\times(-\infty,a]\cup(-\infty,a]\times[-a,+\infty),\\
&\lim_{x\rightarrow+\infty}\sup_{t\leq t_{0}}|u(x,t)-1|=0,~~~
\lim_{x\rightarrow-\infty}\sup_{t\geq t_{0}}|u(x,t)-1|=0,\\
&\lim_{t\rightarrow+\infty}\sup_{x\in(-\infty,N]}|u(x,t)-1|=0.
\end{align*}
Similar statement to (a) also hold for this case, more precisely, the two waves $\phi$ and $\hat{\phi}$ propagating in the same direction and from right of the x-axis as $t\rightarrow-\infty$, and  they getting closer and closer as time goes on, and finally $\hat{\phi}$ is likely to catch up $\phi$.
\item[(c)] If $c>0>\hat{c}$, then $\frac{\partial u}{\partial t}(x,t)>0$ and $u(x,t)$ is increasing with respect to $\theta$ and
\begin{align*}
&\lim_{\theta\rightarrow+\infty}u(x,t)=1  \text{~~uniformly for~~} (x,t)\in[-a,+\infty)^{2}\cup(-\infty,a]\times[-a,+\infty),\\
&\lim_{x\rightarrow\pm\infty}\sup_{t\geq t_{0}}|u(x,t)-1|=0,~~~
\lim_{t\rightarrow+\infty}\sup_{x\in\mathbb{R}}|u(x,t)-1|=0.
\end{align*}
Here the two waves $\phi$ and $\hat{\phi}$ propagating from the opposite ends of the x-axis as $t\rightarrow-\infty$, and are likely to merging with each other eventually.
\end{description}
\end{theorem}

The remainder of this paper is organized as follows. In Section 2, we give some useful lemmas which mainly focus on the initial value problem of \eqref{101} and the speeds of traveling wave solutions. Sections 3 and 4 respectively devote to prove the exponential behaviors of traveling wave solutions and the existence and qualitative properties of entire solution of \eqref{101}.

\section{Preliminaries}
\noindent

In this section, we will make some preparations for getting our main results later. Since the main theorem is proved by considering the solving sequences of Cauchy problems starting at time $-n$ with suitable initial values, we first consider the following Cauchy problem of \eqref{101}:
\begin{equation}\label{201}
\left\{
\begin{aligned}
&\frac{\partial u(x,t)}{\partial t}=(J*u)(x,t)-u(x,t)+f(u(x,t)),
&&x\in\mathbb{R}, t>0,\\
&u(x,0)=u_{0}(x), &&x\in\mathbb{R}.
\end{aligned}
\right.
\end{equation}

\begin{definition}\label{def2.1}
A function $\bar{u}(x,t)$ is called a supersolution of \eqref{101}
on $(x,t)\in\mathbb{R}\times[\tau,T)$, $\tau<T$, if
$\bar{u}(x,t)\in C^{0,1}(\mathbb{R}\times[\tau,T),\mathbb{R})$ and
satisfies
\begin{equation}\label{202}
\frac{\partial}{\partial t}\bar{u}(x,t)\geq (J*\bar{u})(x,t)-\bar{u}(x,t)
+f(\bar{u}(x,t)),~~~~\forall~(x,t)\in\mathbb{R}\times[\tau,T).
\end{equation}
Furthermore, if for any $\tau<T$, $\bar{u}$ is a supersolution of \eqref{101} on $(x,t)\in\mathbb{R}\times[\tau,T)$, then $\bar{u}$ is called a supersolution of \eqref{101} on $(x,t)\in\mathbb{R}\times(-\infty,T)$. Similarly, a subsolution $\underline{u}(x,t)$ can be defined by reversing the
inequality \eqref{202}.
\end{definition}

\begin{lemma}\label{lem2.2}
Assume (J1)-(J2) and (FI) hold. Then\\
{\rm(i)} For any $u_{0}(x)\in C(\mathbb{R},[0,1])$, \eqref{201} admits a unique solution $u(x,t;u_{0})\in C^{0,1}(\mathbb{R}\times[0,\infty),[0,1])$.\\
{\rm(ii)} For any pair of supersolution $\bar{u}(x,t)$ and subsolution
$\underline{u}(x,t)$ of \eqref{101} on $\mathbb{R}\times[0,+\infty)$
with $\underline{u}(x,0)\leq\bar{u}(x,0)$ and $0\leq\underline{u}(x,t),
\bar{u}(x,t)\leq1$ for $(x,t)\in\mathbb{R}\times[0,+\infty)$, there
holds $0\leq\underline{u}(x,t)\leq\bar{u}(x,t)\leq1$ for all
$(x,t)\in\mathbb{R}\times[0,+\infty)$.
\end{lemma}

\begin{lemma}\label{lem2.3}
Assume (J1) and (FI) hold. Let $u(x,t)$ be a solution of \eqref{201} with $u_{0}(x)\in C(\mathbb{R},[0,1])$, then there exists a constant $K_{0}>0$, independent of $x,t$ and $u_{0}(x)$, such that
$$
|u_{t}(x,t)|,~ |u_{tt}(x,t)|\leq K_{0}  \text{~~for any~~} x\in\mathbb{R},~t>0.
$$
In addition, if we further assume $\max_{u\in[0,1]}f'(u)<1$ and there exists a constant $K_{1}>0$ such that for any $\eta>0$,
\begin{equation}\label{eq2.4}
\int_{\mathbb{R}}|J(x+\eta)-J(x)|dx\leq K_{1}\eta,
\end{equation}
and
\begin{equation}\label{2.4}
|u_{0}(x+\eta)-u_{0}(x)|\leq K_{1}\eta,
\end{equation}
then for any $x\in\mathbb{R}$, $t>0$ and $\eta>0$, one has
\begin{align}
|u(x+\eta,t)-u(x,t)|\leq K_{2}\eta,~~~
\left|\frac{\partial u}{\partial t}(x+\eta,t)-\frac{\partial u}
{\partial t}(x,t)\right|\leq K_{2}\eta, \label{2.5}
\end{align}
where $K_{2}$ is some positive constant independent of $u_{0}$ and $\eta$.
\end{lemma}

We omit the details of the proofs of Lemmas \ref{lem2.2} and \ref{lem2.3} because they are standard and common, refer to \cite[Theorem 2.4]{LSW} and \cite[Lemma 3.2]{LZZ}.

Next, we discuss the sign and size of the wave speeds. Due to the asymmetry of the kernel function $J$, the speed $c$ of the nondecreasing traveling wave solution $\phi$ may be nonpositive and the speed $\hat{c}$ of the nonincreasing one $\hat{\phi}$ may be nonnegative, which are very different with the case of the symmetric kernel function. Precisely, integrating the first equation of \eqref{102} from $-\infty$ to $+\infty$, we have
\begin{align}\label{eq3.1}
c&=\int_{\mathbb{R}}\int_{\mathbb{R}}J(y)[\phi(\xi-y)-\phi(\xi)]dyd\xi
+\int_{\mathbb{R}}f(\phi(\xi))d\xi\nonumber\\
&=\int_{\mathbb{R}}\int_{\mathbb{R}}J(y)\int_{0}^{1}\phi'(\xi-\theta y)(-y)d\theta dyd\xi+\int_{\mathbb{R}}f(\phi(\xi))d\xi\nonumber\\
&=-\int_{\mathbb{R}}J(y)ydy+\int_{\mathbb{R}}f(\phi(\xi))d\xi.
\end{align}
Similarly,
\begin{equation}\label{eq302}
\hat{c}=-\int_{\mathbb{R}}J(y)ydy-\int_{\mathbb{R}}f(\hat{\phi}(\hat{\xi}))d\hat{\xi}.
\end{equation}

If $J$ is symmetric, then $\int_{\mathbb{R}}J(y)ydy=0$, it is obvious that $c>0$ and $\hat{c}<0$ since $0\leq\phi(\cdot),\hat{\phi}(\cdot)\leq1$ and $f(u)\geq0$ for $u\in[0,1]$. However, if $J$ is asymmetric, the sign and size of the integral $\int_{\mathbb{R}}J(y)ydy$ can not be determined exactly, thus the signs of speeds $c$ and $\hat{c}$ are uncertain. Actually, from \eqref{eq3.1} and \eqref{eq302}, we know that the signs of $c$ and $\hat{c}$ partially depend on the properties of the kernel function $J$.
More specifically, we will give an example of the asymmetric kernel functions to illustrate such dependence or influence.

\begin{example}\label{em2.1}{\rm
The kernel function given below can make $c\neq0$ and $\hat{c}\neq0$.
\begin{equation*}
J(x)=\left\{
\begin{aligned}
~&\frac{2}{15}e^{-(x-2)}, ~~~~~x\geq2,\\
~&\frac{8}{15}e^{2(x+1)}, ~~~~~~x\leq-1,\\
~&\frac{4}{45}x^{2}-\frac{2}{9}x+\frac{2}{9}, ~~-1<x<2.
\end{aligned}\right.
\end{equation*}
Through some calculations, we verify that $J(x)$ satisfies (J1)-(J2) and $\int_{\mathbb{R}}J(x)xdx=0$, then \eqref{eq3.1} and \eqref{eq302} show that $c\neq0$ and $\hat{c}\neq0$ hold (more precisely, such kernel function can ensure $c>0$ and $\hat{c}<0$). In fact, more generally speaking, in the case of this paper, as long as the kernel function $J(x)$ satisfies (J1)(J2) and $\int_{\mathbb{R}}J(y)ydy\neq\int_{\mathbb{R}}f(\phi(\xi))d\xi$ and $\int_{\mathbb{R}}J(y)ydy\neq\int_{\mathbb{R}}f(\hat{\phi}(\hat{\xi}))d\hat{\xi}$, the wave speeds $c\neq0$ and $\hat{c}\neq0$ are naturally established. One of the simplest case is that $J$ satisfies $\int_{\mathbb{R}}J(y)ydy=0$ besides (J1) and (J2).
}
\end{example}


Now we give a lemma to classify the sigh of the wave speeds.

\begin{lemma}\label{lem2.5}
Assume (J1) and (FI) hold. As for the signs of speeds $c$ and $\hat{c}$, there are only three possibilities:\\
{\rm(i)} $c>0$, $\hat{c}>0$;\\
{\rm(ii)} $c>0$, $\hat{c}<0$;\\
{\rm(iii)} $c<0$, $\hat{c}<0$,\\
which means the case of $c<0, \hat{c}>0$ is impossible. Moreover, $c>\hat{c}$.
\end{lemma}
\begin{proof}
If $c<0$, \eqref{eq3.1} implies $\int_{\mathbb{R}}f(\phi(\xi))d\xi<\int_{\mathbb{R}}J(y)ydy$. Then from
\eqref{eq302}, we get
$$
\hat{c}<-\int_{\mathbb{R}}f(\phi(\xi))d\xi-\int_{\mathbb{R}}f(\hat{\phi}(\hat{\xi}))d\hat{\xi}.
$$
Combine \eqref{eq3.1}, \eqref{eq302} with the properties of $\phi(\cdot)$, $\hat{\phi}(\cdot)$ and $f\mid_{[\rho,1]}>0$, we conclude that $\hat{c}<0$. Similarly, from \eqref{eq3.1}, \eqref{eq302}, one has
\begin{align*}
c-\hat{c}=\int_{\mathbb{R}}f(\phi(\xi))d\xi+\int_{\mathbb{R}}f(\hat{\phi}(\hat{\xi}))d\hat{\xi}
\end{align*}
which yields $c>\hat{c}$.
\end{proof}

\section{Asymptotic behaviors of traveling wave solutions}
\noindent

In the front part of this subsection, we prove the following asymptotic behaviors of traveling waves.
\begin{theorem}\label{th2}
Assume (J1)-(J2) and  (FI) hold. Let $\phi(x+ct)$ and $\hat{\phi}(x+\hat{c}t)$ be the nondecreasing and nonincreasing solutions of \eqref{102} and \eqref{104}, respectively. Then the following conclusions hold:
\begin{description}
\item[(i)] As for $\phi(x+ct)$, one has:
\begin{description}
\item[(a)] There exist two positive constants $A_{0}$ and $\mu_{1}$ such that
$$
0<\phi'(\xi)\leq A_{0}e^{\mu_{1}\xi} \text{~~for~~} \xi\leq0.
$$
\item[(b)] There exist positive constants $A_{2}\leq A_{1}$ and $\mu_{2}\geq\mu_{1}$ such that
$$
A_{2}e^{\mu_{2}\xi}\leq\phi(\xi)\leq A_{1}e^{\mu_{1}\xi} \text{~~for~~} \xi\leq0.
$$
\item[(c)] There exist two constants $A_{3}>0$ and $\mu_{21}<0$ such that
$$
\lim_{\xi\rightarrow+\infty}(1-\phi(\xi))e^{-\mu_{21}\xi}=A_{3}, ~~
\lim_{\xi\rightarrow+\infty}\phi'(\xi)e^{-\mu_{21}\xi}=-A_{3}\mu_{21}.
$$
\end{description}
\item[(ii)] As for $\hat{\phi}(x+\hat{c}t)$, one has:
\begin{description}
\item[(a)] There exist two positive constants $\hat{A}_{0}$ and $\hat{\mu}_{1}$ such that
$$
|\hat{\phi}'(\hat{\xi})|\leq \hat{A}_{0}e^{-\hat{\mu}_{1}\hat{\xi}} \text{~~for~~} \hat{\xi}\geq0.
$$
\item[(b)] There exist positive constants $\hat{A}_{2}\leq \hat{A}_{1}$ and $\hat{\mu}_{2}\geq\hat{\mu}_{1}$ such that
$$
\hat{A}_{2}e^{-\hat{\mu}_{2}\hat{\xi}}\leq\hat{\phi}(\hat{\xi})\leq \hat{A}_{1}e^{-\hat{\mu}_{1}\hat{\xi}} \text{~~for~~} \hat{\xi}\geq0.
$$
\item[(c)] There exist positive constants $\hat{A}_{3}$ and $\hat{\mu}_{22}$ such that
$$
\lim_{\hat{\xi}\rightarrow-\infty}(1-\hat{\phi}(\hat{\xi}))e^{-\hat{\mu}_{22}\hat{\xi}}=\hat{A}_{3}, ~~
\lim_{\hat{\xi}\rightarrow-\infty}\hat{\phi}'(\hat{\xi})e^{-\hat{\mu}_{22}\hat{\xi}}=-\hat{A}_{3}\hat{\mu}_{22}.
$$
\end{description}
\end{description}
\end{theorem}

To obtain the precise exponential asymptotic behaviors of $\phi(\xi)$ as $\xi\rightarrow+\infty$ and $\hat{\phi}(\hat{\xi})$ as $\hat{\xi}\rightarrow-\infty$ (described as (c) of (i) and (ii) in Theorem \ref{th2}), we use a similar argument with Lemma 3.4 in \cite{ZLWS} since $f'(1)<0$. But as $\xi\rightarrow-\infty$ for $\phi(\xi)$
and $\hat{\xi}\rightarrow+\infty$ for $\hat{\phi}(\hat{\xi})$, we can no longer use the same method due to the fact $f'(0)=0$. Here we adopt a proper comparison principle which can be proved by similar arguments with Theorem 3.1 in \cite{Coville2} or Theorem 1.5.1 in \cite{Coville4} and construct appropriate barrier functions to prove Theorem \ref{th2}. The comparison principle is stated as follows.

\begin{theorem}\label{theorem1}
Assume $J$ satisfies (J1)-(J2), $a(x)\in C(\mathbb{R},[0,+\infty))$ and $b(x)\in L^{\infty}(\mathbb{R})\cap C(\mathbb{R})$.
Let $u$ and $v$ be two smooth functions ($C^{1}(\mathbb{R})$) and $\omega$ is a connected subset of $\mathbb{R}$. Assume that $u$ and $v$
satisfy the following conditions:
\begin{align*}
&Lu=J*u-u-a(x)u-b(x)u'\leq0 \text{~~in~~} \omega\subset\mathbb{R},\\
&Lv=J*v-v-a(x)v-b(x)v'\geq0 \text{~~in~~} \omega\subset\mathbb{R},\\
&u\geq_{\not\equiv} v \text{~~in~~}\mathbb{R}-\omega,\\
&\text{if~~} \omega \text{~~is an unbounded domain, also assume that~~} \lim_{\pm\infty}(u-v)\geq0.
\end{align*}
Then $u\geq v$ in $\mathbb{R}$.
\end{theorem}

Now define four complex functions $F_{i}(\mu)$ and $\hat{F}_{i}(\mu)$ for $i=1,2$ by
\begin{align*}
&F_{1}(\mu)=\int_{\mathbb{R}}J(-y)e^{\mu y}dy-1-c\mu+f'(0)=\int_{\mathbb{R}}J(-y)e^{\mu y}dy-1-c\mu,\\
&F_{2}(\mu)=\int_{\mathbb{R}}J(-y)e^{\mu y}dy-1-c\mu+f'(1),\\
&\hat{F}_{1}(\hat{\mu})=\int_{\mathbb{R}}J(y)e^{\hat{\mu} y}dy-1+\hat{c}\hat{\mu},\\
&\hat{F}_{2}(\hat{\mu})=\int_{\mathbb{R}}J(y)e^{\hat{\mu} y}dy-1+\hat{c}\hat{\mu}+f'(1).
\end{align*}
Then the following lemma holds.
\begin{lemma}\label{lem3.3}
Assume J satisfies (J1) and (J2), then the following conclusions hold.\\
{\rm(i)} The equation $F_{1}(\mu)=0$ (or $\hat{F}_{1}(\hat{\mu})=0$) has a positive real root $\mu_{1}$ (or $\hat{\mu}_{1}$) such that
\begin{equation*}
F_{1}(\mu)~(\text{or~}\hat{F}_{1}(\hat{\mu}))
\left\{
\begin{aligned}
&<0  \text{~~for~~} \mu\in(0,\mu_{1})~(\text{or~}\hat{\mu}\in(0,\hat{\mu}_{1})),\\
&>0  \text{~~for~~} \mu>\mu_{1}~(\text{or~}\hat{\mu}>\hat{\mu}_{1}).
\end{aligned}\right.
\end{equation*}
{\rm(ii)} The equation $F_{2}(\mu)=0$ (or $\hat{F}_{2}(\hat{\mu})=0$) has two real roots $\mu_{21}<0$ (or $\hat{\mu}_{21}<0$) and $\mu_{22}>0$ (or $\hat{\mu}_{22}>0$) such that
\begin{equation}
F_{2}(\mu)~(\text{or~}\hat{F}_{2}(\hat{\mu}))\left\{
\begin{aligned}
&>0 \text{~~~for~~}\mu<\mu_{21}~(\text{or~}\hat{\mu}<\hat{\mu}_{21}),\\
&<0 \text{~~~for~~}\mu\in(\mu_{21},\mu_{22})~(\text{or~}\hat{\mu}\in(\hat{\mu}_{21},\hat{\mu}_{22})),\\
&>0 \text{~~~for~~}\mu>\mu_{22}~(\text{or~}\hat{\mu}>\hat{\mu}_{22}).
\end{aligned}
\right.
\end{equation}
\end{lemma}

\begin{proof}
(i) Direct computations show that
$F_{1}(0)=0$ and
\begin{align*}
F_{1}'(\mu)=\int_{\mathbb{R}}J(-y)ye^{\mu y}dy-c,~~~~~
F_{1}''(\mu)=\int_{\mathbb{R}}J(-y)y^{2}e^{\mu y}dy\geq0.
\end{align*}
From \eqref{eq3.1}, we have
$$
F_{1}'(0)=\int_{\mathbb{R}}J(-y)ydy-\int_{\mathbb{R}}J(-y)ydy
-\int_{\mathbb{R}}f(\phi(\xi))d\xi<0.
$$
Moreover, from (J2) we know that $J(y)\not\equiv 0$ in $\mathbb{R}^{-}$, then
$$
F_{1}(\mu)=\int_{-\infty}^{0}J(-y)e^{\mu y}dy+\int_{0}^{+\infty}J(-y)e^{\mu y}dy-1-c\mu\rightarrow+\infty  \text{~~~as~~} \mu\rightarrow+\infty.
$$
Therefore there exists a positive constant $\mu_{1}$ such that $F_{1}(\mu_{1})=0$.
Other conclusions can be proved similarly, we omit the details.
\end{proof}

\begin{proof}[Proof of Theorem \ref{th2}]
We only prove assertion (i) since (ii) can be proved by similar arguments.

(a) The conclusion $\phi'(\xi)>0$ can be obtained directly from the monotonicity of $\phi$ which proved by Coville \cite{Coville2}. Let $\psi(\xi)=A_{0}e^{\mu_{1}\xi}$ and $L$ be the following integro differential operator:
\begin{equation}\label{eq3.3}
Lu=J*u-u-cu'.
\end{equation}
A quick computation shows that $\psi$ satisfies
\begin{align*}
L\psi(\xi)&=A_{0}\int_{\mathbb{R}}J(y)e^{\mu_{1}(\xi-y)}dy-\psi(\xi)-c\mu_{1}\psi(\xi)\\
&=\psi(\xi)\left(\int_{\mathbb{R}}J(-y)e^{\mu_{1}y}dy-1-c\mu_{1}\right)\\
&=0.
\end{align*}

Recalling from Lemma 2.1 in Coville \cite{Coville3}, we know that $\phi'(\xi)\rightarrow0$ as $\xi\rightarrow\pm\infty$. Then by taking $A_{0}=\|\phi'\|_{\infty}$, we achieve $\psi(\xi)>\phi'(\xi)$ on $\mathbb{R}^{+}$. According to the translation invariance,
without loss of generality, we can assume $\phi(0)=\rho$. Then by the monotonicity of $\phi(\xi)$ and the first equation of \eqref{102}, we have $L\phi=0$ on $\mathbb{R}^{-}$.
Now by using Theorem \ref{theorem1} with $\phi'$ and $\psi$ on $\mathbb{R}^{-}$, we obtain the desired inequality:
\begin{equation}\label{306}
\phi'(\xi)\leq\|\phi'\|_{\infty}e^{\mu_{1}\xi} \text{~~~for~~}\xi\leq0,
\end{equation}
which ends the proof of (a).

(b) Due to  $\phi(-\infty)=0$ and $\phi(+\infty)=1$, one has the following asymptotic behavior of $\phi$ by integrating \eqref{306} from $-\infty$ to $\xi$ with $\xi\leq0$,
\begin{equation*}
\phi(\xi)\leq A_{1}e^{\mu_{1}\xi} \text{~~~for~~~} \xi\leq0,
\end{equation*}
where $A_{1}=\|\phi'\|_{\infty}/\mu_{1}$.

To complete the proof, we just need to establish the inequality below:
\begin{equation}\label{307}
A_{2}e^{\mu_{2}\xi}\leq\phi(\xi)
\end{equation}
for $\xi\leq0$ and some positive constants $A_{2}$ and $\mu_{2}$.

Define a function $\varphi_{1}(\xi,\mu)$ for $\mu>0$ by
\begin{equation}\label{eq3.9}
\varphi_{1}(\xi,\mu):=\left\{
\begin{aligned}
&\frac{\rho}{2}e^{\mu\xi} &\text{~~~for~~} \xi\leq0,\\
&\frac{\rho}{2} &\text{~~~for~~} \xi>0.
\end{aligned}\right.
\end{equation}
Redefine $L$ as the following operator:
$$
Lu=J*u-u-cu'-ku.
$$
Choose $k>0$ such that $-k<\frac{f(u)}{u}<k$ for any $u\in[0,1]$, such a $k$ exists because $f$ is Lipschitz continuous on $[0,1]$. Now we compute $L\varphi_{1}$ for $\xi\leq0$,
\begin{align}
L\varphi_{1}(\xi,\mu)&=\varphi_{1}(\xi,\mu)\left[e^{-\mu\xi}\int_{-\infty}^{\xi}J(y)dy
+\int_{\xi}^{+\infty}J(y)e^{-\mu y}dy-1-c\mu-k\right]\nonumber\\
&\geq\varphi_{1}(\xi,\mu)\left[e^{-\mu\xi}\int_{-\infty}^{\xi}J(y)dy
+\int_{\xi}^{0}J(y)e^{-\mu y}dy-1-c\mu-k\right].\label{110}
\end{align}
Since $supp(J)\not\subset\mathbb{R}^{+}$, we can choose $r$ and $\mu_{2}$ large enough such that
\begin{equation}\label{111}
\int_{-r}^{0}J(y)e^{-\mu_{2}y}dy-1-c\mu_{2}-k\geq0.
\end{equation}
Then from \eqref{110} and \eqref{111}, for $\xi\leq-r$ we get
\begin{equation}
L\varphi_{1}(\xi,\mu_{2})\geq\varphi_{1}(\xi,\mu_{2})\left[e^{-\mu_{2}\xi}
\int_{-\infty}^{\xi}J(y)dy+\int_{-r}^{0}J(y)e^{-\mu_{2}y}dy-1-c\mu_{2}-k\right]\geq0.
\end{equation}

Note that $\phi$ and any translation of $\phi$ satisfy on $\mathbb{R}$,
$$
L\phi=J*\phi-\phi-c\phi'-k\phi=-\left(\frac{f(\phi)}{\phi}+k\right)\phi\leq0.
$$
Assume that $\phi(-r)=\frac{\rho}{2}$, then $\varphi_{1}(\xi,\mu_{2})\leq\frac{\rho}{2}\leq\phi(\xi)$ for $\xi>-r$ since $\phi$ is nondecreasing.
Now by using Theorem \ref{theorem1} with $\phi$ and $\varphi_{1}$ on $(-\infty,-r)$, one has $\varphi_{1}(\xi,\mu_{2})\leq\phi$ on $\mathbb{R}$, then conclusion (b) is proved by taking $A_{2}=\min\{\frac{\rho}{2},A_{1}\}$.

The main idea of the proof of statement (c) and Section 5 of \cite{Coville1} is almost the same except for some detailed estimates, so we will not repeat it here.
The proof is complete.
\end{proof}

Theorem \ref{th2} has given some exponential estimates of $\phi(\xi)$ as $\xi\rightarrow-\infty$ and $\hat{\phi}(\xi)$ as $\xi\rightarrow+\infty$, but in order to obtain the entire solution we desired, we need to further research the asymptotic behaviors of $\phi'(\xi)/\phi(\xi)$ as $\xi\rightarrow-\infty$ and $\hat{\phi}'(\hat{\xi})/\hat{\phi}(\hat{\xi})$ as $\hat{\xi}\rightarrow+\infty$. The main idea comes from Zhang et al. \cite{ZGB} which is used efficiently to investigate the detailed asymptotic behavior of the traveling wave solution for nonlocal dispersal equations with degenerate monostable nonlinearity. However in this paper, due to the asymmetry of the kernel function, we can not determine the signs of $c$ and $\hat{c}$ exactly, which further makes the details more difficult and complex. Firstly, we prove the following essential lemma.
\begin{lemma}\label{lem3.4}
Assume $J$ satisfies (J1)-(J2). Let $c\neq0$ and $B(\xi)$ be a continuous function having finite limits at infinity, $B(\pm\infty):=\lim_{x\rightarrow\pm\infty}B(x)$, and $z(\cdot)$ be a measurable function satisfying
\begin{equation}\label{eq3.13}
cz(x)=\int_{\mathbb{R}}J(y)e^{\int_{x}^{x-y}z(s)ds}dy+B(x),~~x\in\mathbb{R}.
\end{equation}
Then $z$ is uniformly continuous and bounded. Moreover, $\mu^{\pm}=\lim_{x\rightarrow\pm\infty}z(x)$ exist and are real roots of the characteristic equation $c\mu=\int_{\mathbb{R}}J(y)e^{-\mu y}dy+B(\pm\infty)$.
\end{lemma}
\begin{proof}
We discuss it by dividing into two cases: (i) $c>0$; (ii) $c<0$.

(i) $c>0$.
Let $w(x)=e^{mx+\int_{0}^{x}z(s)ds}$ for some constant $m$. Then
\begin{equation}\label{eq3.14}
cw'(x)=(cm+B(x))w(x)+\int_{\mathbb{R}}J(y)e^{my}w(x-y)dy.
\end{equation}
Choose $m=\parallel B(x)\parallel_{L^{\infty}(\mathbb{R})}/c$, we get $w'(x)\geq0$.
Integrating \eqref{eq3.14} over $[x-\delta,x]$, one has
\begin{align*}
cw(x)&\geq c\omega(x)-cw(x-\delta)\\
&\geq \int_{\mathbb{R}}J(y)e^{my}\int_{x-\delta}^{x}w(s-y)dsdy \\
&\geq \delta\int_{\mathbb{R}}J(y)e^{my}w(x-\delta-y)dy\\
&\geq \delta\int_{-\infty}^{-2\delta}J(y)e^{my}w(x-\delta-y)dy\\
&\geq \delta w(x+\delta)\int_{-\infty}^{-2\delta}J(y)e^{my}dy.
\end{align*}
Since $supp(J)\not\subset\mathbb{R}^{+}$ from (J2), we can choose $\delta>0$ such that $\int_{-\infty}^{-2\delta}J(y)e^{my}dy>0$. Let $M_{1}:=\int_{-\infty}^{-2\delta}J(y)e^{my}dy$, assumption (J1) ensures $M_{1}<+\infty$. Then we have
\begin{equation}\label{eq3.15}
\frac{w(x+\delta)}{w(x)}\leq\frac{c}{\delta M_{1}}.
\end{equation}
Similarly, integrating \eqref{eq3.14} over $[x,x+\delta]$, we obtain
\begin{equation*}
cw(x+\delta)\geq c\omega(x+\delta)-cw(x)
\geq \delta\int_{\mathbb{R}}J(y)e^{my}w(x-y)dy
\end{equation*}
which implies
\begin{equation}\label{eq3.16}
c\frac{w(x+\delta)}{w(x)}\geq\delta\int_{\mathbb{R}}J(y)e^{my}\frac{w(x-y)}{w(x)}dy
\end{equation}
Note that
\begin{equation*}
cz(x)=\int_{\mathbb{R}}J(y)e^{my}\frac{w(x-y)}{w(x)}dy+B(x).
\end{equation*}
By combining \eqref{eq3.15} with \eqref{eq3.16}, we obtain $z(\cdot)\in L^{\infty}(\mathbb{R})$.

(ii) $c<0$.  From \eqref{eq3.14} we know that $w'(x)\leq0$.
Integrating \eqref{eq3.14} on $[x,x+\delta]$, one has
\begin{align*}
-cw(x)&\geq c\omega(x+\delta)-cw(x)\\
&\geq \delta\int_{\mathbb{R}}J(y)e^{my}w(x+\delta-y)dy\\
&\geq \delta\int_{2\delta}^{+\infty}J(y)e^{my}w(x+\delta-y)dy\\
&\geq \delta w(x-\delta)\int_{2\delta}^{+\infty}J(y)e^{my}dy.
\end{align*}
Let $M_{2}:=\int_{2\delta}^{+\infty}J(y)e^{my}dy$, from assumptions (J1) and (J2), we get $0<M_{2}<+\infty$. Then
\begin{equation*}
\frac{w(x-\delta)}{w(x)}\leq\frac{-c}{\delta M_{2}}.
\end{equation*}
Similarly, we have
\begin{equation*}
-c\frac{w(x-\delta)}{w(x)}\geq\delta\int_{\mathbb{R}}J(y)e^{my}\frac{w(x-y)}{w(x)}dy
\end{equation*}
which yields
\begin{equation*}
\int_{\mathbb{R}}J(y)e^{my}\frac{w(x-y)}{w(x)}dy\leq\frac{c^{2}}{\delta^{2}M_{2}}.
\end{equation*}
Therefore $z(x)$ is bounded. \eqref{eq3.13} implies that $z$ is uniformly continuous.
The rest conclusion can be proved by using a completely similar argument with that of Proposition 3.7 of \cite{ZGB}, because the rest proof is not depend on the sign of $c$ and the symmetry of the kernel function $J$, so we omit the details. The proof is complete.
\end{proof}

Based on above discussions, we obtain the following lemma.
\begin{lemma}\label{lem3.5}
Let $\phi(\xi)$ and $\hat{\phi}(\hat{\xi})$ be the solutions of \eqref{102} and \eqref{104} respectively. Then they satisfy
\begin{equation}\label{eq3.17}
\lim_{\xi\rightarrow-\infty}\frac{\phi'(\xi)}{\phi(\xi)}=\{0,\mu_{1}\},~~
\lim_{\hat{\xi}\rightarrow+\infty}\frac{\hat{\phi}'(\hat{\xi})}{\hat{\phi}(\hat{\xi})}=\{0,-\hat{\mu}_{1}\},
\end{equation}
where $\mu_{1}$ and $\hat{\mu}_{1}$ are defined in Lemma \ref{lem3.3}.
\end{lemma}
\begin{proof}
From Theorem \ref{th2}, one knows that $\phi(\xi)>0$ for any $\xi\in\mathbb{R}$, thus we can define
$$
z(\xi):=\frac{\phi'(\xi)}{\phi(\xi)}.
$$
Dividing the first equation of \eqref{102} by $\phi(\xi)$, one has
\begin{equation*}
cz(\xi)=\int_{\mathbb{R}}J(y)e^{\int_{\xi}^{\xi-y}z(s)ds}dy+B(\xi),
\end{equation*}
where $B(\xi)=-1+\frac{f(\phi(\xi))}{\phi(\xi)}$. Then the front part of \eqref{eq3.17} follows from $B(-\infty)=-1$ and Lemmas \ref{lem3.3} and \ref{lem3.4}. The conclusion for $\hat{\phi}(\hat{\xi})$ can be discussed similarly.
\end{proof}

\section{Entire solution}
\noindent

In this section, we focus our attention on new entire solution of \eqref{101} except traveling wave solutions by combining nondecreasing traveling wave $\phi(x+ct)$ with nonincreasing one $\hat{\phi}(x+\hat{c}t)$. In order to construct
a proper supersolution, we require $\phi$ and $\hat{\phi}$ satisfy the following condition:
\begin{align}\label{eq4.1}
k\phi(\xi)\leq\phi'(\xi) ~~\text{for~~}\xi\leq0~~\text{~and~~~}
\hat{\phi}'(\hat{\xi})\leq-\hat{k}\hat{\phi}(\hat{\xi})~~\text{for~~}\hat{\xi}\geq0
\end{align}
with some positive constants $k$ and $\hat{k}$.
Actually, the assumption \eqref{eq4.1} is easy to satisfy,
for example, from Lemma \ref{lem3.5}, if $\phi$ and $\hat{\phi}$ satisfy
\begin{equation*}
\lim_{\xi\rightarrow-\infty}\frac{\phi'(\xi)}{\phi(\xi)}=\mu_{1},~~
\lim_{\hat{\xi}\rightarrow+\infty}\frac{\hat{\phi}'(\hat{\xi})}{\hat{\phi}(\hat{\xi})}=-\hat{\mu}_{1},
\end{equation*}
then \eqref{eq4.1} can be ensured by combining Theorem \ref{th2}. The simplest case is that the kernel function $J$ has compact support and the radius of the compact support is very small, such as $\text{supp}(J)\subset[-a,b]$ with $0<a,b\ll1$. Without lose of generality, assume $a<b$, note that $\phi(\xi)$ is smooth, then the Taylor's formula yields that
\begin{align*}
J*\phi(\xi)-\phi(\xi)&=\int_{\mathbb{R}}J(y)[\phi(\xi-y)-\phi(\xi)]dy \nonumber\\
&=\frac{1}{2}\int_{\mathbb{R}}J(y)y^{2}dy\phi''(\xi)-\int_{\mathbb{R}}J(y)ydy\phi'(\xi)+o(b^{2})
\text{~~~~~as~~}b\rightarrow0.
\end{align*}
Then for $\xi\ll-1$ and $b\ll1$ small enough, from \eqref{102} and (FI), one has
$$
c\phi'(\xi)=\frac{1}{2}\int_{\mathbb{R}}J(y)y^{2}dy\phi''(\xi)-\int_{\mathbb{R}}J(y)ydy\phi'(\xi)
$$
which implies
$$
\lim_{\xi\rightarrow-\infty}\frac{\phi'(\xi)}{\phi(\xi)}=\lim_{\xi\rightarrow-\infty}\frac{\phi''(\xi)}{\phi'(\xi)}
=\frac{2[c+\int_{\mathbb{R}}J(y)ydy]}{\int_{\mathbb{R}}J(y)y^{2}dy}\neq0.
$$
Then combine with Lemma \ref{lem3.5}, we get $\underset{\xi\rightarrow-\infty}\lim\frac{\phi'(\xi)}{\phi(\xi)}=\mu_{1}$.
Similarly, $\underset{\hat{\xi}\rightarrow+\infty}\lim\frac{\hat{\phi}'(\hat{\xi})}{\hat{\phi}(\hat{\xi})}=-\hat{\mu}_{1}$.

From assumption (FI), we can modify $f(u)$ on $u\in(1,+\infty)$ such that $f'(u)<f'(0)=0$ for any $u\in(1,2)$. Then by the continuity of $f'(u)$ and $f'(1)<0$, there exists $m_{0}\in(0,1)$ so that
\begin{equation}\label{eq4.2}
f'(u)<f'(0)  \text{~~for any~~} u\in(1-m_{0},2).
\end{equation}
In the reminder of this paper, we always assume $\phi(\xi)$ and $\hat{\phi}(\hat{\xi})$ satisfy
\begin{align}
\phi(\xi)\geq1-m_{0} \text{~~for any~~}\xi\geq0,\label{eq4.3}\\
\hat{\phi}(\hat{\xi})\geq1-m_{0} \text{~~for any~~}\hat{\xi}\leq0.\label{eq4.4}
\end{align}
Indeed, \eqref{eq4.3} and \eqref{eq4.4} can be ensured by translating $\phi(\xi)$ and $\hat{\phi}(\hat{\xi})$ along $x$-axis appropriately.

We start with the following ordinary differential problem which plays an important role in the construction of supersolution:
\begin{equation}\label{eq4.5}
\begin{cases}
p'(t)=c_{0}+Ne^{\sigma p(t)},~~~t<0,\\
p(0)<0,
\end{cases}
\end{equation}
where $c_{0}, N$ and $\sigma$ are positive constants and will be chosen later. \eqref{eq4.5} can be solved explicitly as
\begin{equation*}
p(t)-c_{0}t-\omega=-\frac{1}{\sigma}\ln\left\{
1-\frac{r}{1+r}e^{c_{0}\sigma t}\right\},~~r=\frac{N}{c_{0}}e^{\sigma p(0)},
\end{equation*}
with
\begin{equation}\label{406}
\omega:=p(0)-\frac{1}{\sigma}\ln\left\{1+\frac{N}{c_{0}}
e^{\sigma p(0)}\right\}.
\end{equation}
Then $p(t)\leq0$ for $t\leq0$ and there exists some positive constant $K$ such that
\begin{equation}\label{eq4.6}
0<p(t)-c_{0}t-\omega\leq Ke^{c_{0}\sigma t} \text{~~~for~~} t\leq0.
\end{equation}

Now we are ready to construct a supersolution of \eqref{101}, Lemma \ref{lem2.5} shows that $c>\hat{c}$.
\begin{lemma}\label{lem4.1}
Assume that (J1)-(J2) and (FI) hold. Let $\phi(x+ct)$ and $\hat{\phi}(x+\hat{c}t)$ be the traveling wave solutions of \eqref{101} satisfying \eqref{102}, \eqref{104} and \eqref{eq4.1}, respectively.  Further set $\bar{c}=\frac{c+\hat{c}}{2}$, $c_{0}=\frac{c-\hat{c}}{2}$. Then for the solution $p(t)$ of \eqref{eq4.5} with $N>N^{*}$ (it will be given later) and $\sigma:=\min\{\mu_{1},\hat{\mu}_{1}\}$, the function
$$
\bar{u}(x,t)=\phi(x+\bar{c}t+p(t))+\hat{\phi}(x+\bar{c}t-p(t))
$$
is a supersolution of \eqref{101} on $t\in(-\infty, 0]$.
\end{lemma}

\begin{proof}
Define
\begin{equation*}
\mathcal{L}(u)(x,t):=u_{t}(x,t)-(J*u-u)(x,t)-f(u(x,t)).
\end{equation*}
The remain work is to verify $\mathcal{L}(\bar{u})(x,t)\geq0$ for $(x,t)\in\mathbb{R}\times(-\infty,0]$. For simplification, we write $\phi(x+\bar{c}t+p(t))$ and $\hat{\phi}(x+\bar{c}t-p(t))$ as $\phi$ and $\hat{\phi}$, respectively. Direct calculations show that
\begin{align}\label{eq4.7}
\mathcal{L}(\bar{u})&=(\bar{c}+p')\phi'+(\bar{c}-p')\hat{\phi}'-(J*\phi-\phi)-(J*\hat{\phi}-\hat{\phi})
-f(\phi+\hat{\phi})\nonumber\\
&=(\bar{c}+p'-c)\phi'+(\bar{c}-p'-\hat{c})\hat{\phi}'-[f(\phi+\hat{\phi})-f(\phi)-f(\hat{\phi})]\nonumber\\
&=(\phi'-\hat{\phi}')[Ne^{\sigma p}-F(x,t)]
\end{align}
where
$$
F(x,t)=\frac{f(\phi+\hat{\phi})-f(\phi)-f(\hat{\phi})}{\phi'-\hat{\phi}'}.
$$
Next, we study it by dividing $\mathbb{R}$ into three regions.

(i) $p(t)\leq x+\bar{c}t\leq-p(t)$. Then $x+\bar{c}t+p(t)\leq0$ and $x+\bar{c}t-p(t)\geq0$ for $t\leq0$.
Recalling that $\phi,\hat{\phi}\in[0,1]$, $f(0)=0$ and $f\in C^{2}(\mathbb{R})$, one has
$$
f(\phi+\hat{\phi})-f(\phi)-f(\hat{\phi})=\int_{0}^{1}f'(\phi+s\hat{\phi})\hat{\phi}ds
-\int_{0}^{1}f'(s\hat{\phi})\hat{\phi}ds\leq L\phi\hat{\phi},
$$
with $L:=\underset{s\in[0,2]}\max f''(s)$. From Theorem \ref{th2} and \eqref{eq4.1}, for $0\leq x+\bar{c}t\leq-p(t)$, we obtain
\begin{align}\label{eq4.8}
F(x,t)&\leq\frac{L\phi\hat{\phi}}{\phi'-\hat{\phi}'}
\leq\frac{L\hat{\phi}}{\phi'/\phi}\leq\frac{L\hat{A}_{1}e^{-\hat{\mu}_{1}(x+\bar{c}t-p(t))}}{k}
\leq \frac{L\hat{A}_{1}}{k}e^{\hat{\mu}_{1}p(t)}.
\end{align}
Similarly, for $p(t)\leq x+\bar{c}t\leq0$, we get
\begin{align}\label{eq4.9}
F(x,t)\leq\frac{L\phi}{-\hat{\phi}'/\hat{\phi}}\leq\frac{LA_{1}e^{\mu_{1}(x+\bar{c}t+p(t))}}
{\hat{k}}\leq \frac{LA_{1}}{\hat{k}}e^{\mu_{1}p(t)}.
\end{align}

(ii) $-\infty<x+\bar{c}t\leq p(t)$. Combining \eqref{eq4.2} and \eqref{eq4.4} with $x+\bar{c}t-p(t)\leq0$ for $t\leq0$, one has
\begin{align*}
f(\phi+\hat{\phi})-f(\phi)-f(\hat{\phi})&=\int_{0}^{1}f'(\hat{\phi}+s\phi)\phi ds
-\int_{0}^{1}f'(s\phi)\phi ds\\
&\leq\phi\int_{0}^{1}|f'(0)-f'(s\phi)| ds
\leq L\phi^{2}.
\end{align*}
Then
\begin{align}\label{eq4.10}
F(x,t)\leq\frac{L\phi^{2}}{\phi'-\hat{\phi}'}
\leq\frac{L\phi}{\phi'/\phi}
\leq\frac{LA_{1}e^{\mu_{1}(x+\bar{c}t+p(t))}}{k}\leq\frac{LA_{1}}{k}e^{\mu_{1}p(t)}.
\end{align}

(iii) $-p(t)\leq x+\bar{c}t\leq+\infty$. Similarly, from \eqref{eq4.2}, \eqref{eq4.3} and $x+\bar{c}t+p(t)\geq0$, we have $f(\phi+\hat{\phi})-f(\phi)-f(\hat{\phi})\leq L\hat{\phi}^{2}$ and
\begin{align}\label{eq4.11}
F(x,t)\leq\frac{L\hat{\phi}^{2}}{\phi'-\hat{\phi}'}\leq\frac{L\hat{\phi}}{-\hat{\phi}'/\hat{\phi}}
\leq\frac{L\hat{A}_{1}e^{-\hat{\mu}_{1}(x+\bar{c}t-p(t))}}{\hat{k}}\leq\frac{L\hat{A}_{1}}{\hat{k}}e^{\hat{\mu}_{1}p(t)}.
\end{align}
Now by taking
$$
N\geq N^{*}:= \max\left\{\frac{L\hat{A}_{1}}{k},\frac{LA_{1}}{\hat{k}},\frac{LA_{1}}{k},\frac{L\hat{A}_{1}}{\hat{k}}\right\}, ~~~ \sigma:=\min\left\{\mu_{1},\hat{\mu}_{1}\right\}
$$
and combining \eqref{eq4.8}, \eqref{eq4.9}, \eqref{eq4.10}, \eqref{eq4.11} with \eqref{eq4.7}, we conclude that
\begin{equation*}
\mathcal{L}(\bar{u})=(\phi'-\hat{\phi}')(Ne^{\sigma p}-F(x,t))\geq0.
\end{equation*}
The proof is complete.
\end{proof}

\begin{proof}[\textbf{Proof of Theorem \ref{th3}}]
For any $n\in \mathbb{N}$, consider the following Cauchy problem
\begin{equation}\label{eq4.12}
\begin{cases}
\frac{\partial u_{n}}{\partial t}(x,t)=(J*u_{n})(x,t)-u_{n}(x,t)+f(u_{n}(x,t)),~~~(x,t)\in\mathbb{R}\times(-n,+\infty),\\
u_{n}(x,-n):=\underline{u}(x,-n)=\max\{\phi(x-cn+\omega),\hat{\phi}(x-\hat{c}n-\omega)\},~~~x\in\mathbb{R}.
\end{cases}
\end{equation}
Lemma \ref{lem2.2} shows that \eqref{eq4.12} has a unique classical solution $u_{n}(x,t)$ satisfying
\begin{equation}\label{eq4.14}
\underline{u}(x,t)=\max\{\phi(x+ct+\omega),\hat{\phi}(x+\hat{c}t-\omega)\}\leq u_{n}(x,t)\leq1
\end{equation}
for any $(x,t)\in\mathbb{R}\times[-n,+\infty)$ and $n\in\mathbb{N}$.
Furthermore,
$$
u_{n}(x,-n)=\max\{\phi(x-cn+\omega),\hat{\phi}(x-\hat{c}n-\omega)\}\leq
\phi(x-\bar{c}n+p(-n))+\hat{\phi}(x-\bar{c}n-p(-n))
$$
for any $x\in\mathbb{R}$. Then it follows from Lemma \ref{lem2.2} and Lemma \ref{lem4.1} that $u_{n}(x,t)\leq \bar{u}(x,t)$ for any $n\in\mathbb{N}$ and $(x,t)\in\mathbb{R}\times[-n,0]$. Moreover, $|\phi'|\leq\frac{2+M}{|c|}$ and $|\hat{\phi}'|\leq\frac{2+M}{|\hat{c}|}$ with $M:=\max_{u\in[0,1]}f(u)$, then according to Lemma \ref{lem2.3} and Arzela-Ascoli Theorem, there exists a function $u(x,t)$ and a subsequence $\{u_{n_{i}}(x,t)\}$ of $\{u_{n}(x,t)\}$ such that $u_{n_{i}}(x,t)$ and $\frac{\partial }{\partial t}u_{n_{i}}(x,t)$ converge uniformly in any compact set $S\subset\mathbb{R}^{2}$ to $u(x,t)$ and $\frac{\partial}{\partial t}u(x,t)$ as $i\rightarrow+\infty$, respectively. From the equation satisfied by $u_{n_{i}}(x,t)$, it is clearly that $u(x,t)$ is an entire solution of \eqref{101} and satisfies
\begin{align*}
&\underline{u}(x,t)\leq u(x,t)\leq\bar{u}(x,t) \text{~~~for~~}(x,t)\in\mathbb{R}\times(-\infty,0].\\
&\underline{u}(x,t)\leq u(x,t)\leq 1 \text{~~~for~~}(x,t)\in\mathbb{R}^{2}.
\end{align*}
Moreover, by using Theorem \ref{th2} and \eqref{eq4.6}, for $x\geq-\bar{c}t$ and $t\leq0$, one has
\begin{align*}
&u(x,t)-\phi(x+ct+\omega)\\
\leq\,&\phi(x+\bar{c}t+p(t))+\hat{\phi}(x+\bar{c}t-p(t))-\phi(x+ct+\omega)\\
\leq\,&\sup_{\xi\in\mathbb{R}}|\phi'(\xi)|(p(t)-c_{0}t-\omega)+\hat{A}_{1}e^{-\hat{\mu}_{1}(x+\bar{c}t-p(t))}\\
\leq\,& K\sup_{\xi\in\mathbb{R}}|\phi'(\xi)|e^{c_{0}\sigma t}+\hat{A}_{1}e^{\hat{\mu}_{1}p(t)}.
\end{align*}
Similarly, for $x\leq-\bar{c}t$ and $t\leq0$, one has
$$
u(x,t)-\hat{\phi}(x+\hat{c}t-\omega)\leq K\sup_{\hat{\xi}\in\mathbb{R}}|\hat{\phi}'(\hat{\xi})|e^{c_{0}\sigma t}+A_{1}e^{\mu_{1}p(t)}.
$$
Then \eqref{eq4.6} implies \eqref{105} of Theorem \ref{th3}. The inequalities \eqref{106} can be proved directly by Lemma \ref{lem2.3}.

Next we prove the entire solution $u(x,t)$ is increasing with respect to $\omega$.
Note that the traveling waves $\phi$ is nondecreasing and $\hat{\phi}$ is nonincreasing, it follows that the functions $u_{n}(x,-n)$ are nondecreasing in $\omega$ for any $n\in\mathbb{N}$, when the other parameters being fixed. Then $u_{n}(x,t)$, even $u(x,t)$ is nondecreasing in $\omega$. Furthermore, they are increasing in $\omega$ from the strong maximum principle established in \cite{Coville2}.

The entire solution established above is only for the case $\theta=\omega$ with $\omega$ defined by \eqref{406}.
For more general $\theta\in\mathbb{R}$, we define $\tilde{u}(x,t)=u(x+x_{0},t+t_{0})$ with
$$
x_{0}=\frac{(c+\hat{c})(\omega-\theta)}{c-\hat{c}},~~~t_{0}=\frac{2(\theta-\omega)}{c-\hat{c}}.
$$
Denote $\tilde{u}(x,t)$ by $u(x,t)$, then $u(x,t)$ is the entire solution we desired. The rest of the proof is straightforward and mainly depends on the properties of the subsolution $\underline{u}(x,t)$ defined in \eqref{eq4.14}, thus we omit the details.
The proof is complete.
\end{proof}

\section{Discussion}
\noindent

In this paper, we have obtained a new entire solution of nonlocal dispersal equation \eqref{101} with asymmetric kernel function and ignition nonlinearity, but we only consider the interactions of the traveling wave solutions with non-zero speeds. However, from the discussion of Section 2, we know that there are special dispersal kernel functions and nonlinearities can make $c=0$ or $\hat{c}=0$, when at least one of $c$ and $\hat{c}$ is equal to zero, what will occur? If there are some new entire solutions? We guess that there might exist some entire solution that comes from the interactions of two steady state waves or one traveling wave with non-zero speed and one steady state wave, while it seems very different and difficult to mathematically prove this, we leave it as a further investigation.

\section*{Acknowledgments}

\noindent

We are very grateful to the referees for their valuable comments. The second author is partially supported by the NSF of China (11671180) and the Fundamental Research Funds for the Central Universities (lzujbky-2016-ct12) and the third author is partially supported by the NSF of China (11371179).


\begin{thebibliography}{99}

\bibitem{Coville1}
Coville J, D\'{a}vila J, Mart\'{\i}nez S.
Nonlocal anisotropic dispersal with monostable nonlinearity.
{J. Differential Equations}, 2008, 244: 3080--3118%

\bibitem{Coville2}
Coville J. Maximum principles, sliding techniques and applications to nonlocal equations.
{Electron. J. Differential Equations}, 2007, 68: 1--23%

\bibitem{Coville3}
Coville J. Traveling fronts in asymmetric nonlocal reaction diffusion equation: The bistable and ignition case. Pr\'{e}publication du CMM, Hal-00696208, 2012

\bibitem{Coville4}
Coville J. Travelling waves in a nonlocal reaction diffusion equation with ignition nonlinearity [Ph.D. Thesis]. Paris: Universit'e Pierre et Marie Curie, 2003

\bibitem{Crooks}
Crooks E C M , Tsai J C.
Front-like entire solutions for equations with convection.
{J. Differential Equations}, 2012, 253: 1206--1249%

\bibitem{GM}
Guo J S, Morita Y.
Entire solutions of reaction-diffusion equations and an application to discrete diffusive equations.
{Discrete Contin. Dyn. Syst}, 2005, 12: 193--212%

\bibitem{Hamel}
Hamel F, Nadirashvili N.
Entire solution of the KPP eqution.
{Comm. Pure Appl. Math}, 1999, 52: 1255--1276%

\bibitem{liu1}
Li W T, Liu N W, Wang Z C.
Entire solutions in reaction-advection-diffusion equations in cylinders.
{J. Math. Pures Appl}, 2008, 90: 492--504%

\bibitem{LSW}
Li W T, Sun Y J, Wang Z C.
Entire solutions in the Fisher-KPP equation with nonlocal dispersal.
{Nonlinear Anal. Real World Appl}, 2010, 11: 2302--2313%

\bibitem{li2008}
Li W T, Wang Z C, Wu J H. Entire solutions in monostable reaction-diffusion equations with delayed nonlinearity. {J. Differential Equations}, 2008, 245: 102--129%


\bibitem{LZZ}
Li W T, Zhang L, Zhang G B.
Invasion entire solutions in a competition system with nonlocal dispersal.
{Discrete Contin. Dyn. Syst}, 2015, 35: 1531--1560%


\bibitem{MN}
Morita Y, Ninomiya H.
Entire solutions with merging fronts to reaction-diffusion equations.
{J. Dynam. Differential Equations}, 2006, 18: 841--861%


\bibitem{Pan2009}
Pan S X, Li W T, Lin G. Travelling wave fronts in nonlocal delayed reaction-diffusion systems and applications. {Z. Angew. Math. Phys}, 2009, 60: 377--392%

\bibitem{SLW}
Sun Y J, Li W T, Wang Z C.
Entire solutions in nonlocal dispersal equations with bistable nonlinearity.
{J. Differential Equations}, 2011, 251: 551--581%

\bibitem{SLW2}
Sun Y J, Li W T, Wang Z C.
Traveling waves for a nonlocal anisotropic dispersal equation with
monostable nonnlinearity. {Nonlinear Anal}, 2011, 74: 814--826%

\bibitem{SZLW}
Sun Y J, Zhang L, Li W T, Wang Z C.
Entire solutions in nonlocal monostable equations: asymmetric case. Submitted, 2016

\bibitem{wang2010}
Wang M X, Lv G Y. Entire solutions of a diffusive and competitive Lotka-Volterra type system with nonlocal delays. {Nonlinearity}, 2010, 23: 1609--1630%

\bibitem{WLR2016}Wang Z C, Li W T, Ruan S G, Existence, uniqueness and stability of pyramidal traveling fronts in reaction-diffusion systems. {Sci China Math}, 2016, 59: 1869-1908%

\bibitem{WR}
Wu S L, Ruan S G.
Entire solutions for nonlocal dispersal equations with spatio-temporal delay: Monostable case. {J. Differential Equations}, 2015, 258: 2435--2470%


\bibitem{Ya2}
Yagisita H. Existence and nonexistence of traveling waves for a nonlocal monostable equation. {Publ. Res. Inst. Math. Sci}, 2009, 45: 925--953%

\bibitem{ZLW}
Zhang L, Li W T, Wu S L. Multi-type entire solutions in a nonlocal
dispersal epidemic model. {J. Dynam. Differential Equations}, 2016, 28: 189--224%

\bibitem{ZLWS}
Zhang L, Li W T, Wang Z C, Sun Y J. Entire solutions in nonlocal bistable equations: asymmetric case. Submitted, 2016

\bibitem{ZGB}
Zhang G B, Li W T, Wang Z C.
Spreading speeds and traveling waves for nonlocal dispersal equations with degenerate monostable nonlinearity.
{J. Differential Equations}, 2012, 252: 5096--5124%


\end{thebibliography}
\end{document}